\documentclass[12pt]{amsart}
\usepackage[english]{babel}
\usepackage{graphicx}
\usepackage{amsmath,amssymb,srcltx}
\usepackage[hypertex]{hyperref}
\newcommand{\diff}[2]{\mbox{{\rm Diff}{${\,}_{#1}({\mathbb C}^{#2},0)$}}}
\newcommand{\diffr}[2]{\mbox{{\rm Diff}{${\,}_{#1}({\mathbb R}^{#2},0)$}}}
\newcommand{\diffh}[2]{\mbox{$\widehat{\rm Diff}{{\,}_{#1}({\mathbb C}^{#2},0)}$}}
\newcommand{\ndf}[2]{\mbox{$\widehat{\rm End}{{\,}_{#1}({\mathbb C}^{#2},0)}$}}
\newcommand{\diffrh}[2]{\mbox{$\widehat{\rm Diff}{{\,}_{#1}({\mathbb R}^{#2},0)}$}}

\newcommand{\cn}[1]{\mbox{(${\mathbb C}^{#1},0$)}}
\newcommand{\rn}[1]{\mbox{(${\mathbb R}^{#1},0$)}}

\newcommand{\Xn}[2]{\mbox{${\mathcal X}_{#1}\cn{#2}\,$}}
\newcommand{\Xrn}[2]{\mbox{${\mathcal X}_{#1}\rn{#2}\,$}}
\newcommand{\Xf}[2]{\mbox{$\hat{\mathcal X}_{#1}\cn{#2}\,$}}
\newcommand{\Xr}[2]{\mbox{$\hat{\mathcal X}_{#1}\rn{#2}\,$}}

\newtheorem{pro}{Proposition}[section]
\newtheorem{teo}{Theorem}[section]

\newtheorem{cor}{Corollary}[section]

\newtheorem{lem}{Lemma}[section]
\newtheorem{rem}{Remark}[section]
\newtheorem{defi}{Definition}[section]
\newtheorem*{conj}{Conjecture}

\begin{document}

\title[Embedding smooth and formal diffeomorphisms]
{Embedding smooth and formal diffeomorphisms through the Jordan-Chevalley decomposition}

\author{Javier Rib\'{o}n}
\address{Instituto de Matem\'{a}tica, UFF, Rua M\'{a}rio Santos Braga S/N
Valonguinho, Niter\'{o}i, Rio de Janeiro, Brasil 24020-140}
\thanks{e-mail address: javier@mat.uff.br}
\thanks{MSC-class. Primary: 34C20, 37F75; Secondary: 34C41, 34M25}
\date{\today}
\maketitle

\bibliographystyle{plain}
\section*{Abstract}
In  [Xiang Zhang, The embedding flows of $C^{\infty}$
hyperbolic diffeomorphisms, J. Differential Equations
250 (2011), no. 5, 2283-2298] Zhang proved that any local smooth hyperbolic
diffeomorphism whose eigenvalues are weakly nonresonant is embedded
in the flow of a smooth vector field.
We present a new, simpler and more conceptual  proof
of such result using the Jordan-Chevalley
decomposition in algebraic groups and the properties of the
exponential operator.

We characterize the hyperbolic smooth (resp. formal) diffeomorphisms that
are embedded in a smooth (resp. formal) flow.
We introduce a criterium showing that the
presence of weak resonances for a diffeomorphism
plus two natural conditions imply that
it is not embeddable. This solves a conjecture of Zhang.
The criterium is optimal, we provide a method to construct
embeddable diffeomorphisms with weak resonances if we remove
any of the conditions.
\section{Introduction}
We are interested on studying embedding flows for real analytic,
complex analytic and $C^{\infty}$ local diffeomorphisms.

We denote by $\Xrn{\infty}{n}$ and $\diffr{\infty}{n}$
the $C^{\infty}$ local singular vector fields and diffeomorphisms
respectively defined in a neighborhood of $0 \in {\mathbb R}^{n}$.

We denote by $\Xrn{}{n}$
(resp. ${\mathcal X} \cn{n}$) the set of germs of real analytic
(resp. complex analytic) vector fields which are singular at $0$.
The formal completion of these spaces are denoted by
$\Xr{}{n}$ and $\hat{\mathcal X} \cn{n}$ respectively. Indeed
a formal vector field $X \in \Xf{}{n}$ is an expression of the form
\[ \sum_{j=1}^{n} a_{j}(z_{1},\hdots,z_{n}) \frac{\partial}{\partial z_{j}}
\ {\rm where} \ a_{1},\hdots,a_{n} \in {\mathfrak m}
\]
and ${\mathfrak m}$ is the maximal ideal of ${\mathbb C}[[z_{1},\hdots,z_{n}]]$.
Moreover $X$ belongs to $\Xr{}{n}$ if and only if
all the coefficients of the power series $a_{1}$, $\hdots$, $a_{n}$
are real.

We define $\diffr{}{n}$ (resp. ${\rm Diff} ({\mathbb C}^{n},0)$)
the group of local real analytic (resp. complex analytic)
diffeomorphisms defined in a neighborhood of $0 \in {\mathbb C}^{n}$.
We denote by $\diffrh{}{n}$ and  $\diffh{}{n}$ respectively their formal completions.
A formal diffeomorphism $\varphi \in \diffh{}{n}$ is an expression of the form
\[ (a_{1}(z_{1},\hdots,z_{n}), \hdots,
a_{n}(z_{1},\hdots,z_{n}))
\ {\rm where} \ a_{1},\hdots,a_{n} \in {\mathfrak m}
 \]
such that its first jet is an invertible linear operator.
The set $\diffh{}{n}$ is a group for the composition.
The composition in
$\diffh{}{n}$ is defined in the natural way by taking the composition in
$\diff{}{n}$ and passing to the limit in the Krull topology
(see \cite{Eisenbud}, page 204).

We say that $\varphi \in \diffr{\infty}{n}$
(resp. $\diffr{}{n}$, $\diff{}{n}$, $\diffrh{}{n}$, $\diffh{}{n}$) has an
{\it embedding flow} if
there exists $X \in \Xrn{\infty}{n}$
(resp. $\Xrn{}{n}$, $\Xn{}{n}$, $\Xr{}{n}$, $\Xf{}{n}$)
such that $\varphi = {\rm exp}(X)$, i.e. $\varphi$ is the $1$ time flow
of $X$. This concept is defined even if $X$ is formal, in fact
${\rm exp}(X)$ is a formal diffeomorphism such that
$j^{k} {\rm exp}(X) = j^{k} {\rm exp}(X_{k})$ for any $k \in {\mathbb N}$
where $X_{k}$ is an analytic
vector field such that $j^{k} X = j^{k} X_{k}$.

The embedding flow problem is classical.
For instance the embedding flow problem has been deeply studied for $1$-dimensional real
diffeomorphisms (see \cite{Belitskii-Tkachenko} \cite{Beyer-Channell} \cite{Lam-diff} \cite{Lam-hom}).
Palis proved for arbitrary dimension
that the $C^{1}$ diffeomorphisms in a compact manifold that are embedded in a
$C^{1}$ flow form a meagre set \cite{Palis-few}.


Let $\varphi \in \diff{}{}$ be a one variable complex analytic diffeomorphism
such that $j^{1} \varphi \neq Id$. The embedding flow problem is
equivalent to the linearization problem. Indeed $\varphi$
is embedded in  a local complex analytic flow
if and only $\varphi$ is analytically linearizable.
Moreover $\varphi$
has a formal embedding flow
in $\Xf{}{}$ if and only if $\varphi$ is formally linearizable.
In the case $j^{1} \varphi =Id$ and $\varphi \neq Id$ the
diffeomorphism $\varphi$ is always embedded in a formal flow
whereas it has an analytic embedding flow if and only if
the Ecalle-Voronin invariants of $\varphi$
are trivial \cite{V}.

Even in the one dimensional case there are consequences regarding
integrability of complex analytic foliations.
Given a complex analytic codimension $1$ foliation ${\mathcal F}$
and a leaf ${\mathcal L}$ we can associate to ${\mathcal L}$ its
holonomy group ${\mathcal H}$.
It can be interpreted as a subgroup of $\diff{}{}$.
The integrability of the foliation is related to the solvable nature of
these holonomy groups \cite{Paul:pre}. The existence of embedding flows,
in the solvable case, for the elements of the group is related to
the existence of analytic first integrals, integrating factors...
In a different but analogous context the existence of
embedding flows has been applied to find analytic inverse integrating
factors in the neighborhood of limit cycles and elementary singular points
of real analytic planar vector fields \cite{Enciso-Peralta}.

Our point of view in the embedding flow problem
is based on taking profit of the Jordan-Chevalley decomposition
in algebraic groups.
More precisely any
$\varphi  \in \diffh{}{n}$ can be written uniquely
in the Jordan multiplicative
form $\varphi=\varphi_{s} \circ \varphi_{u}= \varphi_{u} \circ \varphi_{s}$
where $\varphi_{s}, \varphi_{u} \in \diffh{}{n}$,
$\varphi_{s}$ is formally conjugated to a diagonal linear transformation
(semisimple part) and
$j^{1} \varphi_{u}$ is a unipotent linear operator (unipotent part).
Analogously any $X \in \Xf{}{n}$ can be written in a unique way in
the Jordan additive form $X = X_{s} + X_{N}$ where
$X_{s},X_{N} \in \Xf{}{n}$,
$X_{s}$ is formally conjugated to a diagonal linear vector field (semisimple part),
$j^{1} X_{N}$ is a nilpotent linear operator (nilpotent part) and
$[X_{s},X_{N}]=0$.
A positive outcome of the decomposition is that we
obtain a natural normal form for $\varphi$ by linearizing $\varphi_{s}$.
Moreover, since affine algebraic groups contain the semisimple and unipotent parts
of all their elements (see 15.3, page 99 \cite{Humphreys})
we can use the Jordan-Chevalley decomposition to study
invariant structures by the action of a diffeomorphism.
For instance given $f \in {\mathbb C}[[z_{1},\hdots,z_{n}]]$ the
set $G=\{ \varphi \in \diffh{}{n} : f \circ \varphi = f \}$ is a group
defined by algebraic equations on the coefficients of $\varphi$.
Hence $\varphi \in G$ implies $\varphi_{s} \in G$ and $\varphi_{u} \in G$.
This is a simplification since $\varphi_{s}$ is formally linearizable and
the properties of $\varphi_{u}$ can be interpreted on terms of the properties
of a formal vector field (the so called infinitesimal generator).
This perspective was used to study invariant and periodic (invariant for an iterate)
analytic and formal curves by elements of $\diff{}{2}$ \cite{michigan}.

Let us focus on the embedding problem for elements $\varphi$ of $\diffrh{}{n}$.
Consider a real linear vector field $X_{1} \in \Xrn{}{n}$ such that
$j^{1} \varphi = {\rm exp}(X_{1})$. We can suppose that $X_{1}$ is in
Jordan normal form. The linear operators $X_{1}$ and $j^{1} \varphi$
have eigenvalues
$\mu_{1}$, $\hdots$, $\mu_{n}$ and
$\lambda_{1}=e^{\mu_{1}}$, $\hdots$, $\lambda_{n}=e^{\mu_{n}}$
respectively.
Consider the Jordan additive (resp. multiplicative) decomposition
$X_{1,s} + X_{1,N}$ (resp. $(j^{1} \varphi)_{s} \circ (j^{1} \varphi)_{u}$)
of $X_{1}$ (resp. of $j^{1} \varphi$).
We say that a monomial $w_{1}^{a_{1}} \hdots w_{n}^{a_{n}} e_{j}$,
where $e_{j}$ is the $j$th element of the canonical base
of ${\mathbb C}^{n}$, is
\begin{itemize}
\item {\it resonant} if $\lambda_{1}^{a_{1}} \hdots \lambda_{n}^{a_{n}} \lambda_{j}^{-1} =1$.
\item {\it strongly resonant} if $a_{1} \mu_{1} + \hdots + a_{n} \mu_{n} - \mu_{j}=0$.
\item {\it weakly resonant} if $a_{1} \mu_{1} + \hdots + a_{n} \mu_{n} - \mu_{j} \in 2 \pi i {\mathbb Z}^{*}$.
\end{itemize}
Resonant implies either strongly or weakly resonant. A monomial
$w_{1}^{a_{1}} \hdots w_{n}^{a_{n}} e_{j}$ is resonant if and only if
it commutes with $(j^{1} \varphi)_{s}$.
Moreover
$w_{1}^{a_{1}} \hdots w_{n}^{a_{n}} e_{j}$ is strongly resonant if and only if
the Lie bracket
$[w_{1}^{a_{1}} \hdots w_{n}^{a_{n}} \partial / \partial {w_{j}}, X_{1,s}]$
is equal to $0$.
Resonances and strong resonances are resonances of the semisimple parts of
$j^{1} \varphi$ and $X_{1}$ respectively.
We say that the eigenvalues of $X_{1}$ are not weakly resonant if there is
no weakly resonant monomial of degree greater or equal than $2$.
In such a case both concepts of resonance coincide.
Zhang proves in this setting that
there is existence and uniqueness of the embedding flow.
\begin{teo}
\label{teo:stunique}
\cite{Zhang}
Let $\varphi \in \diffh{}{n}$
be a formal diffeomorphism.
Let $X_{1}$ be a linear vector field such that $j^{1} \varphi = {\rm exp}(X_{1})$
and whose eigenvalues are not weakly resonant.
Then there exists a unique
$\hat{X}$ in $\Xf{}{n}$ such that
$j^{1} \hat{X} = X_{1}$ and $\varphi = {\rm exp}(\hat{X})$.
Moreover $\hat{X}$ belongs to $\Xr{}{n}$ if
$\varphi \in \diffrh{}{n}$ and $X_{1} \in \Xrn{}{n}$.
\end{teo}
The solution of the embedding flow problem
in the formal setting has implications in the $C^{\infty}$ setting.
The existence of an embedding flow for a hyperbolic element
$\varphi$ of
$\diffr{\infty}{n}$ is equivalent to the existence of an embedding flow
for its asymptotic development $\hat{\varphi} \in \diffrh{}{n}$ at the
origin by a theorem of Chen \cite{Chen}.
\begin{teo}
\label{teo:Zhang}
\cite{Zhang}
Let $\varphi \in {\rm Diff}_{\infty} ({\mathbb R}^{n},0)$
be a hyperbolic diffeomorphism.
Let $X_{1}$ be a linear real vector field such that $j^{1} \varphi = {\rm exp}(X_{1})$
and whose eigenvalues are not weakly resonant.
Then there exists
$X \in {\mathcal X}_{\infty} ({\mathbb R}^{n},0)$ such that
$j^{1} X = X_{1}$ and $\varphi = {\rm exp}(X)$.
\end{teo}
The proof of theorem \ref{teo:stunique}
is obtained by doing an inductive process of
calculations in the jet level.
We introduce a simpler and much more conceptual
proof by using the Jordan-Chevalley
decomposition of vector fields and diffeomorphisms and the
properties of the exponential operator.
Calculations are almost no longer required since they
are encapsulated in the linearization of the
semisimple parts.
Zhang shows that the formal diffeomorphism
can be considered in normal form up to a formal
change of coordinates and then calculates the embedding flow.
The first step can be achieved directly by linearizing the semisimple
part of the formal diffeomorphism. Then it is easy to check out that
$X_{1,s} + \log \varphi_{u}$ is the expression of the embedding flow.
The formal vector field $\log \varphi_{u} \in \Xf{}{n}$ is the
infinitesimal generator of
$\varphi_{u}$, i.e. the unique nilpotent vector field such that
$\varphi_{u} = {\rm exp}(\log \varphi_{u})$.
%
%
%
%

We compare the concepts of embedding flow in $\Xr{}{n}$ and $\Xf{}{n}$
for formal diffeomorphisms in $\diffrh{}{n}$.
They coincide if the linear part of the embedding flow is required to be real.
\begin{teo}
\label{teo:cir}
Let $\varphi \in \diffrh{}{n}$. Suppose that
$\varphi$ is of the form ${\rm exp}(\hat{X})$ for some
$\hat{X} \in \Xf{}{n}$ with $j^{1} \hat{X} \in \Xrn{}{n}$.
Then there exists $\hat{Y} \in  \Xr{}{n}$
such that $\varphi = {\rm exp}(\hat{Y})$ and $j^{1} \hat{Y} = j^{1} \hat{X}$.
\end{teo}

We characterize the diffeomorphisms in $\diffrh{}{n}$ and $\diffh{}{n}$
having an embedding flow via a normal form theorem.
\begin{teo}
\label{teo:embfor}
Let $\varphi \in \diffrh{}{n}$ (resp. $\diffh{}{n}$). Let $X_{1}$
be a linear  element of $\Xrn{}{n}$ (resp. $\Xn{}{n}$) such that
$X_{1,s}$ is diagonal and $j^{1} \varphi = {\rm exp}(X_{1})$.
Then $\varphi$ is embedded
in a formal flow $X$ with $j^{1} X = X_{1}$
if and only if there exists a tangent to the identity
$\eta \in \diffrh{}{n}$ (resp. $\diffh{}{n}$)
such that $\eta^{\circ (-1)} \circ \varphi \circ \eta$ is strongly resonant
(with respect to $X_{1,s}$).
\end{teo}
Let us remind that $\eta \in \diffh{}{n}$ is tangent to the identity if
$j^{1} \eta = Id$. By definition a formal diffeomorphism is strongly
resonant if all its non-vanishing monomials are strongly resonant.

The first examples of hyperbolic diffeomorphisms
$\varphi = Az + O(|z|^{2}) \in \diffr{\infty}{n}$ (resp. $\diffr{}{n}$),
such that $A$ has a real logarithm  $B$,
without embedding vector fields are provided by Zhang \cite{Zhang}.
He conjectures that if $\varphi$
has a non-vanishing weakly resonant monomial then it can not
be embedded in a $C^{\infty}$ (resp. real analytic flow).
We provide a method to build counterexamples to the conjecture
given by resonant diffeomorphisms.
We single out two counterexamples that are particularly relevant.
Each of them implies that an extra condition
should be added to obtain a positive result.
Then we prove that in the new setup the conjecture is true.
As a consequence our examples can be considered as a
classification of the type of counterexamples to the original conjecture.
\begin{teo}
\label{teo:conj}
Let $A \in GL(n,{\mathbb C})$ be a matrix. Let $X_{1}$ be
a logarithm of $A$ such that $X_{1,s}$ is diagonal and
$[X_{1,N}, Y]=0$ for any weakly resonant vector field $Y$.
Consider $\varphi = Az + f_{2} + \hdots + f_{k} + \hdots$ in $\diffr{\infty}{n}$
(resp. $\diffr{}{n}$, $\diffh{}{n}$). Suppose that $\varphi$ satisfies one of the following
conditions:
\begin{itemize}
\item[(a)] $f_{2}=\hdots=f_{k-1}=0$ and $f_{k}$ contains
non-vanishing weakly resonant monomials.
\item[(b)] $Az + f_{2} + \hdots + f_{k-1}$ is strongly resonant,
there is no weakly resonant monomial of degree $2 \leq d \leq k-1$ and
$f_{k}$ contains non-vanishing weakly resonant monomials.
\end{itemize}
Then $\varphi$ is non-embeddable in the
flow of a vector field $X \in \Xrn{\infty}{n}$
(resp. $\Xrn{}{n}$, $\Xf{}{n}$)
such that $j^{1} X = X_{1}$.
\end{teo}
Let us clarify that $f_{j}$ is homogeneous of degree $j$ for $j \geq 2$.
Resonances are considered with respect to $X_{1,s}$.

The condition (a) is weaker than (b) if $f_{2}=\hdots=f_{k-1}=0$.
Otherwise no condition is stronger than the other.

Let us remark that for instance in conditions (a) and (b) of the previous theorem
the diffeomorphism $\varphi$ is not suppposed to be in normal form, or in
other words to be resonant. Moreover we do not suppose in condition (a)
that there are no weakly resonant monomials of degree $2 \leq d \leq k-1$,
or equivalently that there is uniqueness of the embedding flow until order
at most $k-1$.
\section{Real vector fields}
We introduce some useful concepts to study real $C^{\infty}$ or analytic
diffeomorphisms. They include real vector fields, the exponential operator,
the Jordan-Chevalley decomposition of diffeomorphisms, analysis of resonances
and linearization. The results in this section are classical and they are included
for the sake of completeness.
\begin{defi}
We say that $X$ is a formal vector vector field and we denote
$X \in \Xf{}{n}$ if $X$ is a derivation of the ${\mathbb C}$-algebra
${\mathfrak m}$ where ${\mathfrak m}$ is the maximal ideal of
${\mathbb C}[[z_{1},\hdots,z_{n}]]$.
We can express $X \in  \Xf{}{n}$ in the more usual notation
\[ X = X(z_{1}) \frac{\partial}{\partial z_{1}} + \hdots + X(z_{n}) \frac{\partial}{\partial z_{n}}  . \]
We say that $X$ is a holomorphic vector field if
$X({\mathfrak m}_{0}) \subset {\mathfrak m}_{0}$ where
${\mathfrak m}_{0}$ is the maximal ideal of ${\mathbb C}\{z_{1},\hdots,z_{n}\}$
We denote by  $\Xn{}{n}$ the set of local holomorphic vector fields in a neighborhood
of $0$ in ${\mathbb C}^{n}$.
\end{defi}
\begin{defi}
We denote by $\Xrn{\infty}{n}$  the set of $C^{\infty}$ singular vector fields defined in
a neighborhood of $0$ in ${\mathbb R}^{n}$.
We denote by $\diffr{\infty}{n}$  the set of $C^{\infty}$ diffeomorphisms defined in
a neighborhood of $0$ in ${\mathbb R}^{n}$.
\end{defi}
\begin{defi}
\label{def:real}
Let $X \in \Xf{}{n}$ be a formal vector field.
We say that $X$ is real if
\[ \sigma^{*} X = \sigma^{*}
\left( \frac{1}{2} \left( Re (X) - i Im (X) \right) \right) =
\frac{1}{2} \left( Re (X) + i Im (X) \right) \]
for
$\sigma(z_{1},\hdots,z_{n}) = (\overline{z_{1}}, \hdots, \overline{z_{n}})$.
We define $\Xr{}{n}$ the set of real formal vector fields.
We define $\Xrn{}{n}= \Xn{}{n} \cap \Xr{}{n}$.
\end{defi}
A good example is the vector field
$\partial / \partial z= (1/2)(\partial / \partial x - i \partial / \partial y)$.
We have $Re (\partial / \partial z) = \partial / \partial x$ and
$Im (\partial / \partial z) = \partial / \partial y$.
The complex conjugation
$\sigma = \overline{z}$ is of the form $\sigma(x,y) = (x,-y)$ in real coordinates
($z = x + i y$).  Then $\sigma$ preserves $Re (\partial / \partial z)$ whereas it conjugates
$Im(\partial / \partial z)$ and $-Im(\partial / \partial z)$.
The vector field $\partial / \partial z$ is real.
On the contrary
$i \partial / \partial z = (1/2) (\partial / \partial y + i \partial / \partial x)$
is not real since $Re(i \partial / \partial z)=\partial / \partial y$ is not preserved by $\sigma$.
The real vector field $Re (\partial / \partial z)$ preserves the real line ${\mathbb R}$
whereas $Re (i \partial / \partial z)$ does not.
The proof of the next lemma is straightforward and it is omitted.
\begin{lem}
Let $X \in \Xf{}{n}$. The following conditions are equivalent:
\begin{itemize}
\item $X$ is real.
\item $\sigma^{*} Re (X) = Re (X)$.
\item $\sigma^{*} Im (X) = -Im (X)$.
\item $X$ is of the form
\[  \left( \sum a_{j_{1}\dots j_{n}}^{1} z_{1}^{j_{1}} \hdots z_{n}^{j_{n}} \frac{\partial}{ \partial z_{1}}
\right) + \hdots +
\left( \sum a_{j_{1}\dots j_{n}}^{n} z_{1}^{j_{1}} \hdots z_{n}^{j_{n}} \frac{\partial}{\partial z_{n}}
\right) \]
with $a_{j_{1} \hdots j_{n}}^{k} \in {\mathbb R} \ \forall 1\leq k \leq n \ \forall 0 \leq j_{1},\hdots, j_{n}$.
\end{itemize}
\end{lem}
\begin{defi}
We say that $\varphi$ is a formal endomorphism and we denote
$\varphi \in  \ndf{}{n}$ if $\varphi$ is a ${\mathbb C}$-algebra
homomorphism of the maximal ideal of
${\mathbb C}[[z_{1},\hdots,z_{n}]]$.
We can express $\varphi \in  \ndf{}{n}$ in the more usual notation
\[ \varphi = (z_{1} \circ \varphi , \hdots, z_{n} \circ \varphi), \ \
z_{j} \circ \varphi \stackrel{def}{=} \varphi (z_{j}) \ {\rm for} \ 1 \leq j \leq n.
 \]
We say that $\varphi$ is a formal diffeomorphism if
$\varphi$ is an isomorphism.
We say that $\varphi$ is holomorphic if $\varphi({\mathfrak m}_{0}) \subset {\mathfrak m}_{0}$
where ${\mathfrak m}_{0}$ is the maximal ideal of ${\mathbb C}\{z_{1},\hdots,z_{n}\}$.
\end{defi}
\begin{defi}
We denote by  $\diffh{}{n}$ the set of formal diffeomorphisms.
We denote by  $\diff{}{n}$ the set of local holomorphic diffeomorphisms in a neighborhood
of $0$ in ${\mathbb C}^{n}$.
\end{defi}
\begin{defi}
\label{def:rend}
Let $\varphi \in \ndf{}{n}$.   If $\sigma \circ \varphi \circ \sigma = \varphi$
(see def. \ref{def:real})
we say that $\varphi$ is real.
We define $\diffrh{}{n}$ the set of real elements of $\diffh{}{n}$.
We define $\diffr{}{n}= \diff{}{n} \cap \diffrh{}{n}$.
A formal endomorphism $\varphi$ is real
if and only if the formal power series $z_{1} \circ \varphi$, $\hdots$, $z_{n} \circ \varphi$
have real coefficients.
\end{defi}
\begin{rem}
A formal endomorphism $\varphi = (\varphi_{1}, \hdots, \varphi_{n}) \in \ndf{}{n}$ is real
if and only if the formal vector field
$\sum_{j=1}^{n} \varphi_{j} \partial / \partial z_{j}$ is real.
\end{rem}
\subsection{Exponential operator}
Consider a vector field $X \in \Xrn{\infty}{n}$ (or $X \in \Xn{}{n}$).
We denote ${\rm exp}(tX)$ the flow of the vector field $X$, it is
the unique solution
of the differential equation
\[ \frac{\partial}{\partial{t}} {\rm exp}(tX) = X({\rm exp}(tX)) \]
with initial condition ${\rm exp}(0X)=Id$. We define the exponential
${\rm exp}(X)$ of $X$ as
${\rm exp}(1X)$. It is a $C^{\infty}$ local diffeomorphism if $X \in \Xrn{\infty}{n}$.
Moreover ${\rm exp}(X)$ is a holomorphic diffeomorphism if $X \in \Xn{}{n}$.

We can extend the definition of the exponential operator to formal
vector fields as an operator acting on formal power series.
Given  $\hat{X}$ in $\hat{\mathcal X} \cn{n}$ we define
\[
\begin{array}{rccc}
{\rm exp}(\hat{X}): & {\mathbb C}[[z_{1},\hdots,z_{n}]] & \to &
{\mathbb C}[[z_{1},\hdots,z_{n}]] \\
& g & \to & \sum_{j=0}^{\infty} \frac{\hat{X}^{\circ(j)}}{j!} (g)
\end{array}
\]
where $\hat{X}^{\circ(0)}(g)=g$ and $\hat{X}^{\circ(j+1)}(g) = \hat{X}(\hat{X}^{\circ(j)}(g))$
for $j \geq 0$. Notice again that we interpret a formal vector field as a derivation
on the ring of formal power series.
Both definitions of exponential coincide if $\hat{X}$ is
convergent, i.e. $({\rm exp}(\hat{X}))(g) = g \circ {\rm exp}(\hat{X})$ for any
$g \in {\mathbb C}[[z_{1},\hdots,z_{n}]]$. We have
\[ {\rm exp}(\hat{X})(z_{1},\hdots,z_{n}) =
\left(  \sum_{j=0}^{\infty} \frac{\hat{X}^{\circ(j)}}{j!} (z_{1}), \hdots,
 \sum_{j=0}^{\infty} \frac{\hat{X}^{\circ(j)}}{j!} (z_{n}) \right) \]
in the usual notation. Since the coefficients of the exponential series are real
the exponential of a real formal vector field is a real
formal diffeomorphism.
\begin{defi}
Let $\varphi \in \diffr{\infty}{n}$ (resp. $\diffr{}{n}$, $\diff{}{n}$, $\diffh{}{n}$). We say that
$\varphi$ is embedded in a $C^{\infty}$ flow (resp. real analytic , holomorphic, formal flow)
if there exists $X \in \Xrn{\infty}{n}$ (resp. $\Xrn{}{n}$, $\Xn{}{n}$, $\Xr{}{n}$)
such that $\varphi = {\rm exp}(X)$.
\end{defi}
\begin{defi}
Let $X \in \Xf{}{n}$. We say that $X$ is nilpotent if the first jet $j^{1} X$
of $X$ is nilpotent.
We denote by $\Xf{N}{n}$ the set of formal nilpotent vector fields.
\end{defi}
\begin{defi}
Let $\varphi \in \diffh{}{n}$. We say that $\varphi$ is unipotent if $j^{1} \varphi$ is unipotent,
i.e. $1$ is the unique eigenvalue of $j^{1} \varphi$.
We denote by $\diffh{u}{n}$ the set of formal unipotent diffeomorphisms.
\end{defi}
\begin{lem}
\label{lem:exp}
(see \cite{Ecalle}, \cite{MaRa:aen})
The mapping ${\rm exp}: \Xf{N}{n} \to \diffh{u}{n}$ is a bijection.
\end{lem}
\begin{defi}
Let $\varphi \in \diffh{u}{n}$.
The unique nilpotent formal vector field $\log \varphi$ such that
$\varphi = {\rm exp}(\log \varphi)$ is called the infinitesimal generator
of $\varphi$.
\end{defi}
Let us consider $\varphi \in \diffh{u}{n}$ as an operator acting
on power series $Id + \Theta$. More precisely
$\Theta: {\mathbb C}[[z_{1},\hdots,z_{n}]]  \to {\mathbb C}[[z_{1},\hdots,z_{n}]]$
is defined by $\Theta (g) = g \circ \varphi -g$ for any $g \in {\mathbb C}[[z_{1},\hdots,z_{n}]] $.
We have
\begin{equation}
\label{equ:log}
 (\log \varphi)(g) = \log (Id + \Theta) (g)= \sum_{j=1}^{\infty} (-1)^{j+1} \frac{\Theta^{\circ (j)} (g)}{j}
\end{equation}
for any $g \in {\mathbb C}[[z_{1},\hdots,z_{n}]]$.
Since the coefficients of the power series $\log (1 + z)$ are real then
$\log$ associates real formal nilpotent vector fields to real formal unipotent
diffeomorphisms.
\subsection{Jordan-Chevalley decomposition}
Let us recall here some known results \cite{Can-Cer} \cite{MarJ} on the jordanization of
diffeomorphisms and vector fields.

Let ${\mathfrak m}$ the maximal ideal of
${\mathbb C}[[z_{1},\hdots,z_{n}]]$.
Any formal diffeomorphism $\varphi \in \diffh{}{n}$ acts on the
finite dimensional complex vector space ${\mathfrak m}/{\mathfrak m}^{k+1}$ of $k$-jets.
More precisely $\varphi$ defines an element
$\varphi_{k}$ of $GL({\mathfrak m}/{\mathfrak m}^{k+1})$ given by
\begin{equation}
\label{equ:actdif}
\begin{array}{ccc}
{\mathfrak m}/{\mathfrak m}^{k+1} &
\stackrel{\varphi_{k}}{\rightarrow} & {\mathfrak m}/{\mathfrak m}^{k+1} \\
g + {\mathfrak m}^{k+1}& \mapsto & g \circ \varphi + {\mathfrak m}^{k+1}
\end{array} .
\end{equation}
Analogously a formal vector field $X \in \hat{\mathcal X} \cn{n}$
defines an element $X_{k}$ of $GL({\mathfrak m}/{\mathfrak m}^{k+1})$ given by
\begin{equation}
\label{equ:actvf}
\begin{array}{ccc}
{\mathfrak m}/{\mathfrak m}^{k+1} &
\stackrel{X_{k}}{\rightarrow} & {\mathfrak m}/{\mathfrak m}^{k+1} \\
g + {\mathfrak m}^{k+1}& \mapsto & X(g) + {\mathfrak m}^{k+1}
\end{array} .
\end{equation}
Consider the group $D_{k} \subset GL({\mathfrak m}/{\mathfrak m}^{k+1})$
defined as
\[ D_{k} = \{ \alpha \in GL({\mathfrak m}/{\mathfrak m}^{k+1}) :
\alpha (g h) = \alpha (g) \alpha (h) \ \forall g,h \in  {\mathfrak m}/{\mathfrak m}^{k+1} \} .  \]
We define the Lie algebra $L_{k} \subset  End({\mathfrak m}/{\mathfrak m}^{k+1}) $ as
\[ L_{k} = \{ \gamma \in End({\mathfrak m}/{\mathfrak m}^{k+1}) :
\gamma (g h) = \gamma (g) h +  g \gamma (h)
\ \forall g,h \in  {\mathfrak m}/{\mathfrak m}^{k+1} \} .  \]
Any $\gamma \in End({\mathfrak m}/{\mathfrak m}^{k+1})$
admits a  unique additive Jordan decomposition
$\gamma = \gamma_{s} + \gamma_{N}$
where $\gamma_{s}$ is semisimple (or equivalently diagonalizable),
$\gamma_{N}$ is nilpotent and $[\gamma_{s}, \gamma_{N}] =0$.
If $\gamma$ is a derivation, i.e. $\gamma \in L_{k}$, then both
the semisimple and nilpotent parts $\gamma_{s}$ and $\gamma_{N}$
are derivations and belong to $L_{k}$ (see Lemma B, page 18 \cite{Humphreys-lie}).

The equations of the form $\alpha (g h) = \alpha (g) \alpha (h)$ are algebraic
in the coefficients of $\alpha \in GL({\mathfrak m}/{\mathfrak m}^{k+1}) $.
Thus $D_{k}$ is an algebraic group, indeed it is
the subgroup $\{ \varphi_{k} : \varphi \in \diffh{}{n} \}$ of actions on
${\mathfrak m}/{\mathfrak m}^{k+1}$ given by formal diffeomorphisms.
Moreover $\alpha$ admits a unique multiplicative Jordan decomposition
$\alpha = \alpha_{s} \circ \alpha_{u}$
where $\alpha_{s}$ is semisimple,
$\alpha_{u}$ is unipotent and $\alpha_{s} \circ \alpha_{u} = \alpha_{u} \circ \alpha_{s}$.
The Jordan-Chevalley decomposition in algebraic groups implies
$\alpha_{s}, \alpha_{u} \in D_{k}$ (see section 15.3, page 99 \cite{Humphreys}).

An element $\alpha$ of
$D_{k+1}$ satisfies
$\alpha ({\mathfrak m}^{k+1}/ {\mathfrak m}^{k+2}) = {\mathfrak m}^{k+1}/ {\mathfrak m}^{k+2}$.
Therefore $\alpha$ induces a unique element in $D_{k}$. In this way
we  define a morphism $\pi_{k} :D_{k+1}  \to D_{k}$ of algebraic groups.
It satisfies $\pi_{k} (\varphi_{k+1}) = \varphi_{k}$ for all
$\varphi \in \diffh{}{n}$ and $k \in {\mathbb N}$.
The Jordan decomposition is preserved by $\pi_{k}$. More precisely
we obtain $\pi_{k}(\varphi_{k+1,s}) = \varphi_{k,s}$ and
$\pi_{k}(\varphi_{k+1,u}) = \varphi_{k,u}$ for all
$\varphi \in \diffh{}{n}$ and $k \in {\mathbb N}$.
As a consequence there exist $\varphi_{s}, \varphi_{u} \in \diffh{}{n}$
such that $(\varphi_{s})_{k} = \varphi_{k,s}$ and  $(\varphi_{u})_{k} = \varphi_{k,u}$
for  all $\varphi \in \diffh{}{n}$ and $k \in {\mathbb N}$.

Since $(\varphi_{s})_{k}$ is diagonalizable for any $k \in {\mathbb N}$ it can be proved
that $\varphi_{s}$ is formally diagonalizable by induction on $k$
(see lemma \ref{lem:rldif}).
The formal diffeomorphism $\varphi_{u}$ satisfies that $(\varphi_{u})_{k}$ is unipotent
for any $k \in {\mathbb N}$. It is easy to see that this is equivalent to the
unipotency of $j^{1} \varphi_{u} = (\varphi_{u})_{1}$.
The next proposition summarizes the previous discussion.
\begin{pro}
Let $\varphi \in \diffh{}{n}$. Then there exist unique formal diffeomorphisms
$\varphi_{s}, \varphi_{u} \in \diffh{}{n}$ such that
$\varphi = \varphi_{s} \circ \varphi_{u} = \varphi_{u} \circ \varphi_{s}$, $\varphi_{s}$
is formally diagonalizable and $\varphi_{u}$ is unipotent.
\end{pro}
The next result is the analogue for vector fields.
It is obtained by considering the additive Jordan decomposition.
\begin{pro}
Let $X \in \Xf{}{n}$. Then there exist unique formal vector fields
$X_{s}, X_{N} \in \Xf{}{n}$ such that
$X= X_{s} + X_{N}$, $[X_{s},X_{N}]=0$, $X_{s}$
is formally diagonalizable and $X_{N}$ is nilpotent.
\end{pro}
The following results are a direct consequence of the real nature of the Jordan
decomposition.
\begin{lem}
\label{lem:jdvfr}
Let $X \in \Xr{}{n}$.  Then
$X_{s}$ and $X_{N}$ are real.
\end{lem}
\begin{lem}
Let $\varphi \in \diffrh{}{n}$. Then
$\varphi_{s}$ and $\varphi_{u}$ are real.
\end{lem}
\subsection{Real monomials}
\label{subsec:realmon}
Let $\varphi \in \diffrh{}{n}$. We can suppose that $j^{1} \varphi_{s}$
is a diagonal transformation in some coordinates
$(w_{1},\hdots,w_{n})$ of ${\mathbb C}^{n}$
up to a linear change of coordinates.
The components of the diffeomorphism $\varphi$
are not anymore real power series if there exists a complex non-real
eigenvalue of $j^{1} \varphi$. We are interested on working on
coordinates making $\varphi_{s}$ is as simple as possible.
As a consequence it is necessary to characterize the real nature
of endomorphisms and vector fields in the new coordinates.

Fix a real matrix $M = M_{s} + M_{N} \in End({\mathbb R}^{n})$
such that $M_{s}$ is in real Jordan normal form.
For instance such a property holds true if $M$ itself is in
Jordan normal form. The matrix $M_{s}$ is diagonalizable.
Its real Jordan
blocks are of the forms
\begin{equation}
\label{equ:jordan}
J =
\left(
\begin{array}{rr}
\lambda & -\mu \\
\mu & \lambda
\end{array}
\right) \ {\rm for} \ \lambda, \mu \in {\mathbb R} \ {\rm and} \
(\beta) \ {\rm for} \ \beta \in {\mathbb R}.
\end{equation}
Consider coordinates $(z_{p}, z_{p+1})$ in the former case.
The complex Jordan normal form is obtained by considering the
linear change of base
\begin{equation}
\label{equ:rtoc}
\left(
\begin{array}{c}
z_{p} \\
z_{p+1}
\end{array}
\right) =
\left(
\begin{array}{rr}
1 & 1 \\
-i & i
\end{array}
\right)
\left(
\begin{array}{c}
w_{p} \\
w_{p+1}
\end{array}
\right) .
\end{equation}
We define $\rho(p) = p+1$ and  $\rho(p+1) = p$.
If the block is of the form $(\beta)$
in a coordinate $z_{q}$ we define
$w_{q} = z_{q}$ and $\rho(q) =q$.
It is convenient to work in the coordinates $(w_{1},\hdots,w_{n})$ since
the matrix on the left-hand side of expression (\ref{equ:jordan}) becomes
\[
\left(
\begin{array}{cc}
\lambda + i \mu & 0 \\
0 & \lambda - i \mu \\
\end{array}
\right) .
\]
The matrix of the operator $M_{s}$ in coordinates $(w_{1},\hdots,w_{n})$
is diagonal.
We denote $(\gamma_{1},\hdots,\gamma_{n})$ the eigenvalues of the matrix $M$ in
coordinates $(w_{1},\hdots,w_{n})$.
We obtain $\overline{\gamma_{j}} = \gamma_{\rho(j)}$ for any $1 \leq j \leq n$ by construction.
Consider a monomial ${\mathfrak h} = \lambda w_{1}^{a_{1}} \hdots w_{n}^{a_{n}} e_{j}$
where $\lambda \in {\mathbb C}$ and $e_{j}$ is the $j$th element of the canonical base
of ${\mathbb C}^{n}$. We denote
$\overline{\mathfrak h} = \overline{\lambda}
w_{\rho (1)}^{a_{1}} \hdots w_{\rho (n)}^{a_{n}} e_{\rho (j)}$.
We obtain that
\begin{equation}
\label{equ:real}
{\mathfrak h} = \frac{{\mathfrak h} + \overline{\mathfrak h}}{2} + i
\frac{{\mathfrak h} - \overline{\mathfrak h}}{2i}
\end{equation}
is the decomposition in real and imaginary parts of ${\mathfrak h}$
(see def. \ref{def:rend}).
\begin{defi}
Let $M \in End({\mathbb R}^{n})$ be a semisimple matrix.
We say that $M$ is diagonal if it is in real Jordan normal form.
In such a case we consider the coordinates $(z_{1},\hdots,z_{n})$
and  $(w_{1},\hdots,w_{n})$ associated to $M$ above.
\end{defi}
\begin{rem}
Given a matrix $M \in  End({\mathbb R}^{n})$
we consider that it is in normal form if $M_{s}$ is diagonal.
We do not require $M$ to be in Jordan normal form. One reason is that
the condition in the semisimple part is simpler. A deeper reason is
that this choice of normal forms is preserved by the exponential.
More precisely
if $M$ is in real Jordan normal form then ${\rm exp}(M)$
is not necessarily in Jordan normal form whereas if
$M_{s}$ is diagonal then ${\rm exp}(M)_{s}$ is diagonal.
\end{rem}
\begin{defi}
We say that $\lambda w_{1}^{a_{1}} \hdots w_{n}^{a_{n}} e_{j}$
(or $\lambda w_{1}^{a_{1}} \hdots w_{n}^{a_{n}} \partial / \partial w_{j}$)
is monomial of degree $a_{1} + \hdots + a_{n}$.
We say that a polynomial is homogeneous of degree $k$ if it is a sum
of degree $k$ monomials.
\end{defi}
\begin{defi}
We say that $\lambda w_{1}^{a_{1}} \hdots w_{n}^{a_{n}} e_{j}$
(or $\lambda w_{1}^{a_{1}} \hdots w_{n}^{a_{n}} \partial / \partial w_{j}$)
is strongly resonant (with respect to $M$)
if $\gamma_{j} = <\gamma, a> =  \sum_{k=1}^{n} \gamma_{k} a_{k}$.
\end{defi}
\begin{rem}
\label{rem:resvf}
Let ${\mathfrak h} = w_{1}^{a_{1}} \hdots w_{n}^{a_{n}} \partial / \partial w_{j}$.
We have
\begin{equation}
\label{equ:comvf}
\left[ \sum_{k=1}^{n} \gamma_{k} w_{k} \frac{\partial}{\partial w_{k}}, {\mathfrak h} \right]=
( <\gamma, a> - \gamma_{j}){\mathfrak h}.
\end{equation}
Then ${\mathfrak h}$
is strongly resonant if and only if
$[\sum_{k=1}^{n} \gamma_{k} w_{k} \partial / \partial w_{k}, {\mathfrak h}]= 0$.
\end{rem}
\begin{defi}
We say that $\lambda w_{1}^{a_{1}} \hdots w_{n}^{a_{n}} e_{j}$
(or $\lambda w_{1}^{a_{1}} \hdots w_{n}^{a_{n}} \partial / \partial w_{j}$)
is weakly resonant (with respect to $M$)
if $\gamma_{j} - <\gamma, a> \in 2 \pi i {\mathbb Z}^{*}$.
We say that the eigenvalues $(\gamma_{1},\hdots,\gamma_{n})$
are weakly resonant if there exists a weakly resonant monomial of
degree greater than $1$.
\end{defi}
\begin{defi}
We say that $\lambda w_{1}^{a_{1}} \hdots w_{n}^{a_{n}} e_{j}$
(or $\lambda w_{1}^{a_{1}} \hdots w_{n}^{a_{n}} \partial / \partial w_{j}$)
is a resonant monomial
if it is either strongly or weakly resonant. Equivalently the monomial is resonant if
$e^{\gamma_{j}} = (e^{\gamma_{1}})^{a_{1}} \hdots (e^{\gamma_{n}})^{a_{n}}$.
\end{defi}
\begin{rem}
\label{rem:resd}
Let ${\mathfrak h} = w_{1}^{a_{1}} \hdots w_{n}^{a_{n}} e_{j}$.
We have
\[
(e^{\gamma_{1}} w_{1}, \hdots ,e^{\gamma_{n}} w_{n})^{\circ (-1)} \circ {\mathfrak h} \circ
(e^{\gamma_{1}} w_{1}, \hdots ,e^{\gamma_{n}} w_{n}) =
e^{-\gamma_{j}}  (e^{\gamma_{1}})^{a_{1}} \hdots (e^{\gamma_{n}})^{a_{n}} {\mathfrak h} . \]
Then ${\mathfrak h}$
is resonant if and only if it commutes with
$(e^{\gamma_{1}} w_{1}, \hdots ,e^{\gamma_{n}} w_{n})$.
\end{rem}
\begin{defi}
We say that a formal endomorphism (resp. vector field) is resonant
(resp. strongly, weakly resonant, nonresonant)
if all its non-vanishing monomials are resonant
(resp. strongly, weakly resonant, nonresonant).
\end{defi}
The property $\overline{\gamma_{j}} = \gamma_{\rho(j)}$ for any $1 \leq j \leq n$ implies
\begin{lem}
\label{lem:spcon}
We have that
$w_{1}^{a_{1}} \hdots w_{n}^{a_{n}} e_{j}$ is resonant
(resp. strongly, weakly resonant)
if and only if
$w_{\rho (1)}^{a_{1}} \hdots w_{\rho (n)}^{a_{n}} e_{\rho (j)}$ is resonant
(resp. strongly, weakly resonant).
\end{lem}
\begin{lem}
\label{lem:monplay}
Let ${\mathfrak h} = \lambda w_{1}^{a_{1}} \hdots w_{n}^{a_{n}} \partial / \partial w_{j}$
be a strongly resonant monomial.
Let ${\mathfrak k} = \mu  w_{1}^{b_{1}} \hdots w_{n}^{b_{n}} \partial / \partial w_{k}$ be a monomial.
\begin{itemize}
\item $[{\mathfrak h}, {\mathfrak k}]$ is strongly resonant if ${\mathfrak k}$
is strongly resonant.
\item $[{\mathfrak h}, {\mathfrak k}]$ is weakly resonant if ${\mathfrak k}$
is weakly resonant.
\item $[{\mathfrak h}, {\mathfrak k}]$ is nonresonant if ${\mathfrak k}$
is nonresonant.
\end{itemize}
\end{lem}
\begin{lem}
\label{lem:logres}
Let $\varphi \in \diffh{}{n}$ with $j^{1} \varphi ={\rm exp} (X_{1})$ where
$X_{1}$ is a real linear vector field such that $X_{1,s}$ is diagonal.
Suppose that $\varphi$
is strongly resonant (resp. resonant) with respect to $X_{1,s}$. Then
$\log \varphi_{u}$ is strongly resonant (resp. resonant).
\end{lem}
\begin{proof}
We have $X_{1,s} = \sum_{k=1}^{n} \gamma_{k} w_{k} \partial / \partial w_{k}$.
Thus
$\phi= (e^{\gamma_{1}} w_{1}, \hdots, e^{\gamma_{n}} w_{n})$ commutes
with $\varphi$ by remark \ref{rem:resd}.
We have that
\[ j^{1} \varphi = {\rm exp}(X_{1,s}) \circ {\rm exp}(X_{1,N}) =
\phi \circ {\rm exp}(X_{1,N}) \]
is the Jordan decomposition of $j^{1} \varphi$.
As a consequence
$\varphi = \phi \circ (\phi^{\circ (-1)} \circ \varphi)$ is the Jordan-Chevalley decomposition of
$\varphi$. In particular we obtain $\varphi_{u} = \phi^{\circ (-1)} \circ \varphi$ and the
non-zero monomials of $\varphi$ and $\varphi_{u}$ coincide.
The rest of the proof is a simple
calculation based on equation (\ref{equ:log}).
\end{proof}
\subsection{Linearization of vector fields and diffeomorphisms}
Formal semisimple diffeomorphisms and vector fields are formally linearizable.
If they are real we can also choose a real formal diffeomorphism as the linearizing
transformation.
\begin{lem}
\label{lem:rlvf}
Let $X \in \Xr{}{n}$ be a formal semisimple vector field.
Then there exists a formal diffeomorphism $\eta \in \diffrh{}{n}$ such that
$\eta^{*} X = j^{1} X$.
\end{lem}
\begin{lem}
\label{lem:rldif}
Let $\varphi \in \diffrh{}{n}$ be a formal semisimple diffeomorphism.
Then there exists a formal diffeomorphism $\eta \in \diffrh{}{n}$ such that
$\eta^{\circ (-1)} \circ \varphi \circ \eta = j^{1} \varphi$.
\end{lem}
\begin{proof}[proof of lemmas \ref{lem:rlvf} and \ref{lem:rldif}]
Let us show lemma  \ref{lem:rlvf}. The proof of lemma  \ref{lem:rldif} is analogous.
Up to a real linear change of coordinates we can suppose that
$j^{1} X_{s}$ is diagonal in coordinates $(z_{1},\hdots,z_{n})$.
We obtain
$j^{1} X =
\gamma_{1} w_{1} \partial / \partial w_{1} + \hdots \gamma_{n} w_{n} \partial / \partial w_{n}$
in the coordinates $(w_{1},\hdots,w_{n})$ introduced in this section.
Suppose that $X$ is of the form
$j^{1} X + X_{k} + Y_{k+1}$ where $X_{k}$ is homogeneous of degree
$k$ and $Y_{k+1}$ is a sum of monomials of degree greater than $k$.
It suffices to prove that there exists a diffeomorphism
$\eta_{k} \in \diffr{}{n}$
such that $\eta_{k}^{*} X = j^{1} X + X_{k+1} + Y_{k+2}$
where $X_{k+1}$ is homogeneous of degree
$k+1$, $Y_{k+2}$ is a sum of monomials of degree greater than $k+1$
and $\eta_{k} -Id$ is a sum
of monomials of degree greater or equal than $k$.
In this way we obtain
$\eta = \lim_{k \to \infty} \eta_{2} \circ \hdots \circ \eta_{k}$ in $\diffrh{}{n}$ such that
$\eta^ {*} X = j^{1} X$.

The vector field $X_{k}$ is a sum of real vector fields of the form
\[ X_{k,a_{1},\hdots,a_{n},j,\lambda} = {\mathfrak h}_{\lambda} \stackrel{def}{=}
\lambda w_{1}^{a_{1}} \hdots w_{n}^{a_{n}} \frac{\partial}{\partial w_{j}} + \overline{\lambda}
w_{\rho (1)}^{a_{1}} \hdots w_{\rho (n)}^{a_{n}}  \frac{\partial}{\partial w_{\rho(j)}} . \]
Indeed the vector field ${\mathfrak h}_{\mu}$ is real for any $\mu \in {\mathbb C}$
(see equation (\ref{equ:real})).
Suppose that the monomials of $ {\mathfrak h}_{\lambda}$ are not strongly resonant.
Since we have $[{\mathfrak h}_{\mu}, j^{1} X] = {\mathfrak h}_{\mu (\gamma_{j} - <\gamma,a>)}$
we define $\mu  = \lambda/(\gamma_{j} - <\gamma,a>)$
and $\tilde{X}_{k,a_{1},\hdots,a_{n},j,\lambda} = {\mathfrak h}_{\mu}$.
We denote by $X_{k}^{0}$ (resp. $\tilde{X}_{k}^{0}$)
the sum of the non-strongly resonant vector fields of
the form $ X_{k,a_{1},\hdots,a_{n},j,\lambda}$ (resp.
of the form $\tilde{X}_{k,a_{1},\hdots,a_{n},j,\lambda}$).
Consider the real diffeomorphism $\eta_{k} = {\rm exp}(-\tilde{X}_{k}^{0})$. We obtain
\[ (\eta_{k}^{\circ (-1)})^{*} j^{1} X = j^{1} X + [\tilde{X}_{k}^{0}, j^{1} X] + \frac{1}{2!}
[\tilde{X}_{k}^{0}, [\tilde{X}_{k}^{0},j^{1}X]] + \hdots  = \]
\[ = j^{1} X + X_{k}^{0} + h.o.t. \]
Moreover we obtain
$(\eta_{k}^{\circ (-1)})^{*} (j^{1} X  + X_{k} - X_{k}^{0}) = j^{1} X + X_{k} + h.o.t$.
Hence $\eta_{k}^{*} X = j^{1} X  + X_{k} - X_{k}^{0} +h.o.t$.
The vector field $\eta_{k}^{*} X$ is still semisimple.
We obtain
\[ (\eta_{k}^{*} X)_{k} = j^{1} X + (X_{k} - X_{k}^{0}) \]
as the Jordan decomposition (see equation (\ref{equ:actvf}))
of  $(\eta_{k}^{*} X)_{k}$
since $j^{1} X$ is semisimple, $X_{k} - X_{k}^{0}$ is nilpotent and
$[ j^{1} X , X_{k} - X_{k}^{0}]=0$ (see eq. (\ref{equ:comvf})).
Since $(\eta_{k}^{*} X)_{k}$ is semisimple and
the Jordan-Chevalley decomposition is unique we obtain  $X_{k} - X_{k}^{0} \equiv 0$.
\end{proof}
\begin{rem}
\label{rem:rlvf}
Let us notice that a simpler version of the previous proof shows that
a formal semisimple vector field in $\Xf{}{n}$ is formally linearizable.
\end{rem}
\section{Embedding flows}
Given $\varphi \in \diffrh{}{n}$ we can consider whether it is embedded
in the flow of a formal vector field in $\Xf{}{n}$ or $\Xr{}{n}$.
A priori these properties could be different.
Indeed real diffeomorphisms can be embedded in the flows of
non-real vector fields, for example we have
$Id = {\rm exp}(2 \pi i z \partial / \partial z)$.
Since Jordanization interprets elements of $\Xr{}{n}$
as formal vector fields in $\Xf{}{n}$ theorem \ref{teo:cir}
justifies our approach.
\begin{proof}[proof of theorem \ref{teo:cir}]
Let $\hat{X} = X_{s} + X_{N}$ be
the Jordan-Chevalley decomposition of $\hat{X}$.
Since ${\rm exp}(X_{s})$ is semisimple (and then formally linearizable)
and ${\rm exp}(X_{N})$ is unipotent
we obtain $\varphi_{s} = {\rm exp}(X_{s})$ and $\varphi_{u} = {\rm exp}(X_{N})$.
We have $\log \varphi_{u} = X_{N}$ by lemma \ref{lem:exp}.
Indeed $\varphi_{u}$ and $X_{N}$ are real.
The difficulty of the proof is that $X_{s}$ is not necessarily real.

We denote $\alpha = Re (X_{s})$ and $\beta = Im(X_{s})$. In fact
$\alpha$ and $\beta$ are elements of $\Xr{}{n}$ such that
$X_{s} = \alpha + i \beta$. By lemma \ref{lem:jdvfr} it suffices to prove
$\varphi = {\rm exp}(\alpha_{s} + \log \varphi_{u})$
where $\alpha_{s}$ is the semisimple part of $\alpha$.

We have
\[ [X_{s},X_{N}]= 0 \implies [\alpha,X_{N}]=0 \ {\rm and} \ [\beta, X_{N}]=0. \]
Moreover since $j^{1} \hat{X}$, $j^{1} X_{s}$ are real then
$j^{1} \alpha = j^{1} X_{s}$ and $j^{1} \beta = 0$.
We obtain ${\rm exp}(t X_{N})^{*} \alpha = \alpha$ for any $t \in {\mathbb C}$
as a consequence of $[\alpha, X_{N}]=0$. The uniqueness of the Jordan-Chevalley
decomposition
\[ {\rm exp}(t X_{N})^{*} (\alpha_{s} + \alpha_{N}) =
{\rm exp}(t X_{N})^{*} (\alpha_{s}) + {\rm exp}(t X_{N})^{*} (\alpha_{N}) = \alpha_{s} + \alpha_{N} \]
implies ${\rm exp}(t X_{N})^{*} (\alpha_{s}) =\alpha_{s}$ for any $t \in {\mathbb C}$.
We deduce $[\alpha_{s},X_{N}]=0$.
We have $\varphi_{s}^{*} (\hat{X}) = \hat{X}$ since $[X_{s}, \hat{X}]=0$
and $\varphi_{s} ={\rm exp}(X_{s})$.
Thus we obtain
$\varphi_{s}^{*} (X_{s})=X_{s}$ and $\varphi_{s}^{*} (X_{N})=X_{N}$
by uniqueness of the Jordan-Chevalley decomposition.
Since $\varphi_{s}$ is real we have $\varphi_{s}^{*} \alpha = \alpha$ and then
$\varphi_{s}^{*} \alpha_{s} = \alpha_{s}$.
We deduce the equality
$\varphi_{s} \circ {\rm exp}(\alpha_{s}) = {\rm exp}(\alpha_{s}) \circ \varphi_{s}$.
Moreover, $j^{1} \alpha = j^{1} X_{s}$ is semisimple; thus we
get $j^{1} \alpha_{s}=j^{1} X_{s}$.

Let us show that $\varphi_{s} = {\rm exp}(\alpha_{s})$. Since
we have $j^{1} \varphi_{s} = j^{1} {\rm exp}(\alpha_{s})$ then the
formal diffeomorphism $\eta =  {\rm exp}(-\alpha_{s}) \circ \varphi_{s}$ is
unipotent. The diffeomorphisms $\varphi_{s}$ and ${\rm exp}(\alpha_{s})$
commute, we obtain
\[ \varphi_{s} = \varphi_{s} \circ Id = {\rm exp}(\alpha_{s}) \circ \eta \]
two Jordan-Chevalley decompositions of $\varphi_{s}$.
We deduce $\varphi_{s} = {\rm exp}(\alpha_{s})$.
We define $\hat{Y} = \alpha_{s} + \log \varphi_{u}$.
It satisfies $j^{1} \hat{Y} = j^{1} \alpha_{s} + j^{1} X_{N} = j^{1} \hat{X}$.
Then $[\alpha_{s},X_{N}]=0$ implies
${\rm exp}(\hat{Y}) = {\rm exp}(\alpha_{s}) \circ {\rm exp}(\log \varphi_{u}) =
\varphi_{s} \circ \varphi_{u} = \varphi$.
\end{proof}

{\bf Example.}
It is clear that there are elements $\varphi$ of
$\diffrh{}{n}$ that can be embedded in flows in
$\Xf{}{n}$ but not in flows in $\Xr{}{n}$.
An example is provided by
$\varphi = (-z_{1}, z_{1}-z_{2})$ since the
linear operator $\varphi$ is embedded in the flow of
$\pi i z_{1} \partial / \partial z_{1} + (-z_{1} + \pi i z_{2}) \partial / \partial z_{2}$
but not in a real one. Indeed Jordan blocks associated to negative eigenvalues
of real matrices with real logarithms appear pairwise \cite{Culver}.
But even in the class of diffeomorphisms in $\diffrh{}{n}$ whose linear part
has a real logarithm it is possible to find elements
that are embedded in flows in
$\Xf{}{n}$ but not in flows in $\Xr{}{n}$.
Next we introduce an example.

Consider $\mu_{1}=-2+ \sqrt{2} 2 \pi i$,
$\mu_{2}=-2- \sqrt{2} 2 \pi i$, $\mu_{3}=3+ (1-4 \sqrt{2}) \pi i$,
$\mu_{4}=3- (1-4 \sqrt{2}) \pi i$. A simple calculation provides that
the ${\mathbb Z}$-module
\[ V = \{ (m_{1},m_{2},m_{3},m_{4}) \in {\mathbb Z}^{4} :
m_{1} \mu_{1} + m_{2} \mu_{2} + m_{3} \mu_{3} +m_{4} \mu_{4} \in 2 \pi i {\mathbb Z} \} \]
is equal to
${\mathbb Z} (3,3,2,2) + {\mathbb Z} (5,1,3,1)$.
We define $\mu_{1}'= \mu_{1} - 4 \pi i$, $\mu_{2}'=\mu_{2}$,
$\mu_{3}'=\mu_{3}+ 6 \pi i$, $\mu_{4}'=\mu_{4}$. We obtain
\[ m_{1} \mu_{1}' + m_{2} \mu_{2}' + m_{3} \mu_{3}' +m_{4} \mu_{4}' =0 \ \forall
(m_{1},m_{2},m_{3},m_{4}) \in  V . \]
On the contrary we have
\begin{equation}
\label{equ:wrex}
5 \mu_{1}'' + \mu_{2}'' + 3 \mu_{3}'' + \mu_{4}''  \neq 0 .
\end{equation}
for any choice of logarithms
$\mu_{1}''$, $\mu_{2}''$, $\mu_{3}''$, $\mu_{4}''$ of
$e^{\mu_{1}}$, $e^{\mu_{2}}$, $e^{\mu_{3}}$, $e^{\mu_{4}}$
respectively such that $\mu_{1}'' = \overline{\mu_{2}''}$,
$\mu_{3}'' = \overline{\mu_{4}''}$.
Otherwise if $\mu_{1}''= \mu_{1} - 2 \pi i k_{1}$ and
$\mu_{3}''= \mu_{3} - 2 \pi i k_{2}$ we obtain
$4 k_{1} + 2 k_{2} =1$ for some $k_{1}, k _{2} \in {\mathbb Z}$.

Consider the linear diffeomorphism $\varphi_{0}$ defined by
\[
\left( -2 z_{1} -  \pi \sqrt{8} z_{2}, \pi \sqrt{8} z_{1}  -2 z_{2},
3 z_{3} -  (1 - 4 \sqrt{2}) \pi z_{4}, (1 - 4 \sqrt{2}) \pi z_{3} + 3 z_{4} \right). \]
By considering a real logarithm $B$ of $\varphi_{0}$ we can
introduce coordinates $(w_{1},w_{2},w_{3},w_{4})$
such that $\varphi_{0}=
(e^{\mu_{1}} w_{1}, e^{\mu_{2}} w_{2}, e^{\mu_{3}} w_{3}, e^{\mu_{4}} w_{4})$
as in section \ref{subsec:realmon}.
These coordinates do not depend on the choice of the matrix $B$
since the decomposition of ${\mathbb C}^{4}$ as direct sum of eigenspaces of $B$
does not depend on $B$. Indeed it
coincides with the analogous decomposition associated to $\varphi_{0}$.
Consider
\[ \varphi = (e^{\mu_{1}} w_{1} + w_{1}^{6} w_{2}  w_{3}^{3}  w_{4}
, e^{\mu_{2}} w_{2} + w_{1} w_{2}^{6}  w_{3}  w_{4}^{3} ,
e^{\mu_{3}} w_{3}, e^{\mu_{4}} w_{4}) . \]
It is a real hyperbolic element of $\diffr{}{4}$.
We have that $\varphi$ is embedded in a formal flow of linear part
$\sum_{j=1}^{4} \mu_{j}' w_{j} \partial / \partial w_{j}$
by theorem \ref{teo:stunique}.
The monomials $w_{1}^{6} w_{2}  w_{3}^{3}  w_{4} e_{1}$ and
$w_{1} w_{2}^{6}  w_{3}  w_{4}^{3} e_{2}$ are weakly resonant for any
choice of a real logarithm of $\varphi_{0}$ by equation (\ref{equ:wrex}).
Therefore $\varphi$ is not embedded in a flow in $\Xr{}{4}$
by theorem \ref{teo:conj} (remark that $X_{1,N}$ is a vanishing vector field).
Of course $\sum_{j=1}^{4} \mu_{j}' w_{j} \partial / \partial w_{j}$ is not real.
We can enlarge the class of embeddable diffeomorphisms by considering
non real logarithms of the linear part but if the linear part of the logarithm
is real theorem \ref{teo:cir} implies that
we can not enlarge the class of embeddable diffeomorphisms by
trying to consider non real formal flows.

A classical
way of obtaining normal forms for local holomorphic vector fields and diffeomorphisms
is by considering changes of coordinates in which the semisimple part is linear.
We apply this ideas to characterize whether or not a diffeomorphism in
$\diffr{}{n}$ or $\diff{}{n}$ is embedded in a formal flow $X$.
The embeddability of the diffeomorphism is equivalent to the existence
of a strongly resonant normal form.
\begin{proof}[proof of theorem \ref{teo:embfor}]
Let us prove the result for $\varphi \in \diffrh{}{n}$. The proof for
$\varphi$ in $\diffh{}{n}$ is simpler.

Suppose that $\varphi = {\rm exp}(X)$ with $X \in \Xr{}{n}$ and $j^{1} X = X_{1}$.
Consider the Jordan-Chevalley decomposition $X =X_{s} + X_{N}$ of $X$.
We have $j^{1} X_{s} = X_{1,s}$.
The proof of lemma \ref{lem:rlvf} implies the existence of a
tangent to the identity $\eta \in \diffrh{}{n}$ such that
$\eta^{*} X_{s} = j^{1} X_{s}$. We have
\[ \eta^{\circ (-1)} \circ \varphi \circ \eta = {\rm exp}(\eta^{*} X) =
 {\rm exp}(X_{1,s} + \eta^{*} X_{N}) = {\rm exp}(X_{1,s}) \circ {\rm exp}(\eta^{*} X_{N}) . \]
Moreover $\eta^{*} X_{N}$ is strongly resonant since $[X_{1,s},\eta^{*} X_{N}]=0$
(remark \ref{rem:resvf}). Thus ${\rm exp}(\eta^{*} X_{N})$ is strongly resonant.
Since ${\rm exp}(X_{1,s})$ is diagonal then
$\eta^{\circ (-1)} \circ \varphi \circ \eta$ is strongly resonant.

Suppose that $\tilde{\varphi} = \eta^{\circ (-1)} \circ \varphi \circ \eta$ is strongly resonant.
We define $\phi = {\rm exp}(-X_{1,s}) \circ \tilde{\varphi}$.
Since $j^{1} \phi = (j^{1} {\varphi}_{s})^{\circ (-1)} \circ j^{1} \varphi = j^{1} \varphi_{u}$
then $\phi$ is unipotent. Moreover  $\tilde{\varphi}$ and ${\rm exp}(X_{1,s})$ commute;
thus $\tilde{\varphi} = {\rm exp}(X_{1,s}) \circ \phi$ is the Jordan-Chevalley
decomposition of $\tilde{\varphi}$.
We apply lemma \ref{lem:logres} to $\tilde{\varphi}$
to obtain that $\log \phi \in \Xr{}{n}$ is strongly resonant.
This implies the key property
$[X_{1,s}, \log \phi] = 0$ (remark \ref{rem:resvf}).  We denote $X=X_{1,s} + \log \phi$.
We obtain
\[ {\eta}^{\circ (-1)} \circ {\varphi} \circ {\eta} =
{\rm exp}(X_{1,s}) \circ {\rm exp}(\log \phi) =
{\rm exp}(X_{1,s} + \log \phi) . \]
We have $j^{1} (X_{1,s} + \log \phi) = X_{1,s} + j^{1} \log \varphi_{u} =X_{1}$
by lemma \ref{lem:exp}. Thus $\varphi$ is of the form
${\rm exp}(\eta_{*}X)$ where $\eta_{*}X \in \Xr{}{n}$ and $j^{1} \eta_{*}X = X_{1}$.
\end{proof}
\begin{rem}
Consider the case $\varphi \in \diffrh{}{n}$ and $X_{1} \in \Xrn{}{n}$
in theorem \ref{teo:embfor}. Then a normalizing map $\eta \in \diffh{}{n}$
implies that $\varphi$ has an embedding flow whose first jet is equal to
$X_{1}$ by theorems \ref{teo:embfor} and \ref{teo:cir}.
\end{rem}
As an application of Jordanization techniques we present
a new, simpler proof of theorem \ref{teo:stunique}.
The idea is that if a diffeomorphism is embedded in a flow $X$ then
the linearization of its semisimple part $X_{s}$ provides a strongly
resonant expression. We can always obtain a resonant expression
of $\varphi$ by linearizing $\varphi_{s}$. The hypothesis implies
that these concepts are the same.
\begin{proof}[proof of th. \ref{teo:stunique}]
Let us suppose that $\varphi$ and $X_{1}$ are real in order to prove the existence.
The general case is simpler.
Up to a real linear change of coordinates we can suppose that $X_{1,s}$ is
diagonal in coordinates $(z_{1},\hdots,z_{n})$. Let
$(w_{1},\hdots,w_{n})$ be the system of coordinates introduced in section
\ref{subsec:realmon}.
In particular $\varphi_{1,s}={\rm exp}(X_{1,s})$ is diagonal.

Consider the Jordan-Chevalley decomposition
\[ {\varphi} = {\varphi}_{s} \circ {\varphi}_{u} = {\varphi}_{u} \circ {\varphi}_{s} \]
of ${\varphi}$. There exists a tangent to the identity ${\eta} \in \diffrh{}{n}$ such that
${\eta}^{\circ (-1)} \circ {\varphi}_{s} \circ {\eta} = j^{1} {\varphi}_{s}$
by lemma \ref{lem:rldif}.
We denote $\tilde{\varphi}= {\eta}^{\circ (-1)} \circ {\varphi} \circ {\eta}$.
Since $\tilde{\varphi}$ commutes with $\tilde{\varphi}_{s}=j^{1} {\rm exp}(X_{1,s})$
then $\tilde{\varphi}$ is resonant (rem. \ref{rem:resd}).
Moreover the properties $[X_{1,s},X_{1}]=0$ and
$j^{1} \tilde{\varphi} = {\rm exp}(X_{1})$ imply that
$j^{1} \tilde{\varphi}$ is strongly resonant.
In particular $\tilde{\varphi}$ is strongly resonant by hypothesis.
There exists
$\hat{X} \in \Xr{}{n}$ with $j^{1} \hat{X} = X_{1}$ such that
$\hat{\varphi} ={\rm exp}(\hat{X})$ (th. \ref{teo:embfor}).
%
%

We have to prove that
${\varphi} = {\rm exp}(\hat{X}) = {\rm exp}(\hat{Y})$ and
$j^{1} \hat{X} = j^{1} \hat{Y} = X_{1}$ imply $\hat{X}=\hat{Y}$.
Let $\hat{X} = X_{s} + X_{N}$, $\hat{Y} = Y_{s} + Y_{N}$ be the
Jordan-Chevalley decompositions of $\hat{X}$, $\hat{Y}$ respectively.
We have
$j^{1} X_{s} = j^{1} Y_{s} = X_{1,s}$,
${\rm exp}(X_{s})={\rm exp}(Y_{s})$ and $X_{N}=Y_{N}=\log {\varphi}_{u}$.
Up to a formal change of coordinates we can suppose that
$X_{s}=j^{1} X_{s}$ (lemma \ref{lem:rlvf} and remark \ref{rem:rlvf}). There exists a formal
tangent to the identity diffeomorphism $\eta \in \diffh{}{n}$ conjugating
$X_{s}=j^{1} X_{s}$ and $Y_{s}$ (remark \ref{rem:rlvf}). We obtain
\[ \eta \circ {\rm exp}(X_{s}) = {\rm exp}(Y_{s}) \circ \eta \implies
\eta \circ j^{1} {\varphi}_{s} =  j^{1} {\varphi}_{s} \circ \eta. \]
Hence $\eta$ is resonant (remark \ref{rem:resd}) and by hypothesis strongly
resonant. Thus we obtain $Y_{s} = \eta_{*} j^{1} X_{s} = j^{1} X_{s} = X_{s}$
and $\hat{X} = \hat{Y}$.
\end{proof}
\begin{proof}[proof of theorem \ref{teo:Zhang}]
We include Zhang's proof for the sake of completeness.
Let $\hat{\varphi} \in \diffrh{}{n}$ be the asymptotic development of $\varphi$.
Consider the element $\hat{X}$ of $\Xr{}{n}$ such that
$\hat{\varphi} = {\rm exp}(\hat{X})$ and $j^{1} \hat{X} = X_{1}$ provided
by th. \ref{teo:stunique}.
Let $X'$ be an element of $\Xrn{\infty}{n}$ whose
asymptotic development at the origin is equal to $\hat{X}$.
The hyperbolic diffeomorphisms $\varphi$ and ${\rm exp}(X')$ are
formally conjugated by the identity.
Two formally conjugated hyperbolic $C^{\infty}$ local diffeomorphisms
are conjugated by a local $C^{\infty}$ diffeomorphism (Chen \cite{Chen}).
There exists $\upsilon \in \diffr{\infty}{n}$
such that $\varphi = \upsilon^{\circ (-1)} \circ {\rm exp}(X') \circ \upsilon$.
We obtain $\varphi = {\rm exp}(X)$ for
$X = {\rm exp}(\upsilon^{*} X') \in \Xrn{\infty}{n}$.
\end{proof}
It is interesting to study to what extent weakly resonances are obstructions
for diffeomorphisms to be embedded. In this spirit
we want to address a conjecture by Zhang.
Let $A$ be a hyperbolic matrix in $GL(n,{\mathbb R})$ and let $B$ be
a real logarithm of $A$.
\begin{conj}
\cite{Zhang}
If $f(z) = O(|z|^{2})$ is $C^{\infty}$ (resp. analytic) and it has
a non-vanishing weakly resonant monomial
$w_{1}^{m_{1}} \hdots w_{n}^{m_{n}} e_{j}$ (with respect to $B$),
then the locally hyperbolic diffeomorphism $\varphi (z) = A z + f(z)$ is not embedded
in a $C^{\infty}$ (resp. analytic) flow.
\end{conj}
The conjecture as stated is false and we provide two counterexamples given by
resonant diffeomorphisms.
\subsection{Building examples}
We explain a method to obtain non-vanishing weakly resonant monomials for embeddable
diffeomorphisms even if the diffeomorphism is resonant.
Let us consider $X_{s}$ a real linear diagonal vector field.
Consider  a real  nilpotent vector field $X_{N}$ such that $[X_{s},X_{N}]=0$.
In particular $X_{N}$ is strongly resonant.
We also suppose that $X_{N}$ is homogeneous of degree $k$.
It is clear that
\[ \varphi = {\rm exp}(X_{s} + X_{N}) = {\rm exp}(X_{s}) \circ {\rm exp}(X_{N})
= {\rm exp}(X_{N}) \circ {\rm exp}(X_{s}) \]
is embeddable in $\diffr{}{n}$. Consider $\eta = Id + \eta_{l} \in \diffr{}{n}$ ($l \geq 2$)
where $\eta_{l}$ is a homogeneous weakly resonant endomorphism of degree $l$.
It is clear that $\eta^{\circ (-1)} \circ \varphi \circ \eta$ is embeddable in $\diffr{}{n}$.
Moreover if $j^{1} \varphi$ is hyperbolic then any $\phi \in \diffr{\infty}{n}$ whose
asymptotic development coincides with $\varphi$ is embeddable by Chen's theorem \cite{Chen}.

We denote $\alpha_{s} = {\rm exp}(X_{s})$. Since $\eta$ is resonant
then  $\eta^{\circ (-1)} \circ \alpha_{s} \circ \eta = \alpha_{s}$
(remark \ref{rem:resd}). Moreover we have
\begin{equation}
\label{equ:taylorvf}
\eta^{*} X_{N} = X_{N} + [\log \eta, X_{N}] + \frac{1}{2!} [\log \eta, [ \log \eta, X_{N}]] + \hdots
\end{equation}
It is natural to try to find a weakly resonant
$Y = w_{1}^{a_{1}} \hdots w_{n}^{a_{n}} \partial / \partial w_{j}$ of degree $l$
such that $[Y, X_{N}] \neq 0$.
Then $[Re (Y) + i Im (Y), X_{N}] \neq 0$ (see eq. (\ref{equ:real}))
implies either $[Re (Y) , X_{N}]  \neq 0$  or $[Im (Y), X_{N}] \neq 0$.
Anyway there exists a real weakly resonant homogeneous vector field $Y'$ of degree $l$
such that $[Y', X_{N}] \neq 0$ (lemma \ref{lem:spcon}).
We define $\eta = j^{l} {\rm exp}(Y')$.
The formula (\ref{equ:taylorvf}) implies that
$\eta^{*} X_{N}$ is of the form $X_{N} + [Y',X_{N}] + Z_{l+k}$
where $[Y',X_{N}]$ is weakly resonant and homogeneous of degree $l+k-1$
by lemma \ref{lem:monplay}
and $Z_{l+k}$ is a sum of monomials of degree greater or equal than $l+k$.
Therefore $\eta^{*} X_{N}$ is not strongly resonant. On the contrary
$\eta^{*} X_{N}$ is resonant since
\[ \alpha_{s}^{*} (\eta^{*} X_{N}) = (\eta \circ \alpha_{s})^{*} X_{N} =
(\alpha_{s} \circ \eta)^{*} X_{N} = \eta^{*}  (\alpha_{s}^{*} X_{N}) = \eta^{*} X_{N} . \]
The second equality is as consequence of the resonant nature of $\eta$.
We deduce that ${\rm exp} (\eta^{*} X_{N})$ is resonant but not strongly resonant by
lemma \ref{lem:logres}. Thus
\[ \eta^{\circ (-1)} \circ \varphi \circ \eta = (\eta^{\circ (-1)} \circ \alpha_{s} \circ \eta)
\circ (\eta^{\circ (-1)} \circ {\rm exp}(X_{N}) \circ \eta) =
\alpha_{s} \circ {\rm exp}(\eta^{*} X_{N}) \]
is embeddable, resonant but not strongly resonant.

Next we provide a condition on the eigenvalues of $X_{s}$ that guarantees
that the previous method can be applied. Roughly speaking the condition is equivalent to
the existence of  infinitely many independent weakly linear monomials.
\begin{lem}
Let $X_{s} \in \Xrn{}{n}$ be a linear diagonal vector field.
Let $0 \neq X_{N} \in \Xrn{}{n}$ be  a homogeneous nilpotent vector field such that
$[X_{s},X_{N}]=0$. Suppose that $X_{N}$ is
strongly resonant.
Suppose that the eigenvalues $\gamma_{1}$, $\hdots$, $\gamma_{n}$ of $X_{s}$
satisfy $\sum_{j=1}^{n} m_{j} \gamma_{j} \in 2 \pi i {\mathbb Z}^{*}$
for some $(m_{1},\hdots,m_{n}) \in ({\mathbb N} \cup \{0\})^{n}$.
Then there exists a weakly resonant monomial vector field $Y$ such that
$[Y,X_{N}] \neq 0$.
\end{lem}
The resonances are considered with respect to $X_{s}$.
\begin{proof}
We denote $X_{N} = \sum_{q=1}^{n} b_{q}(w_{1},\hdots,w_{n}) \partial / \partial w_{q}$.
Every monomial vector field
$W_{k,j} = (w_{1}^{m_{1}} \hdots w_{n}^{m_{n}})^{k} w_{j} \partial / \partial w_{j}$
is weakly resonant for all $k \geq 1$ and $1 \leq j \leq n$.
It suffices to prove that we can not have $[X_{N},W_{k,j}] = 0$ for all
$k \geq 1$ and $1 \leq j \leq n$.

The property $[X_{N},W_{k,q}] (w_{j}) = 0$ for $q \neq j$ implies
$\partial b_{j} / \partial w_{q} = 0$. In particular $b_{j}$ depends only on $w_{j}$
for any $1 \leq j \leq n$.
Let $d$ be the common degree of the polynomials $b_{1}$, $\hdots$, $b_{n}$.
The polynomial $b_{j}$ is of the form $\lambda_{j} w_{j}^{d}$ for some
$\lambda_{j} \in {\mathbb C}$ and any $1 \leq j \leq n$.
We have $d \geq 2$, otherwise we would get $X_{N}=0$ since $X_{N}$ is nilpotent.
The property
$[X_{N},W_{k,j}] (w_{j}) = 0$ implies
\[ \lambda_{j} d  w_{j}^{d} =
\sum_{q \neq j}
\lambda_{q} k m_{q} w_{q}^{d-1}  w_{j} +
\lambda_{j} (k m_{j} + 1)  w_{j}^{d}  \]
for any $k \in {\mathbb N}$.  We deduce $\lambda_{j}=0$ for any $1 \leq j \leq n$.
We obtain $X_{N}=0$ contradicting the hypothesis.
\end{proof}
\subsection{Example}
\label{subsec:ex1}
We consider
\[ X_{s} = -2 z_{1} \frac{\partial}{\partial z_{1}} +
\left( z_{2} - \frac{\pi}{2} z_{3} \right)  \frac{\partial}{\partial z_{2}} +
\left( \frac{\pi}{2}  z_{2} + z_{3} \right)  \frac{\partial}{\partial z_{3}} \]
or
\[ X_{s} = -2 w_{1} \frac{\partial}{\partial w_{1}} +
\left( 1 + \frac{\pi}{2} i \right) w_{2}  \frac{\partial}{\partial w_{2}} +
\left( 1 - \frac{\pi}{2} i \right) w_{3}  \frac{\partial}{\partial w_{3}} \]
in coordinates $(w_{1},w_{2},w_{3})$. We have $\rho (1) =1$, $\rho (2)=3$ and $\rho (3)=2$.
The monomial $X_{N} = w_{1}^{2} w_{2} w_{3} \partial / \partial w_{1}$ is real,
nilpotent and strongly resonant. The monomial $Y = w_{1} w_{3}^{3}  \partial / \partial w_{2}$
is weakly resonant.
We define
$Y' =  w_{1} w_{3}^{3}  \partial / \partial w_{2} +  w_{1} w_{2}^{3}  \partial / \partial w_{3}$.
We obtain
\[ [Y', X_{N}] = (w_{1}^{3} w_{3}^{4} + w_{1}^{3} w_{2}^{4}) \frac{\partial}{\partial w_{1}} -
w_{1}^{2} w_{2} w_{3}^{4}  \frac{\partial}{\partial w_{2}} -
w_{1}^{2} w_{2}^{4} w_{3}  \frac{\partial}{\partial w_{3}}  . \]
We define
\[ \eta = j^{4} {\rm exp}(Y') =
(w_{1}, w_{2} + w_{1} w_{3}^{3}, w_{3} + w_{1} w_{2}^{3}). \]
and ${\varphi} = {\rm exp}(X_{s} + \eta^{*} X_{N})$. We obtain
$\varphi = (\varphi_{1},\varphi_{2},\varphi_{3})$ with
\[  \varphi_{1} = e^{-2} (w_{1} + w_{1}^{2} w_{2} w_{3} +
w_{1}^{3} w_{2}^{4} + w_{1}^{3} w_{3}^{4} + w_{1}^{3} w_{2}^{2} w_{3}^{2}) + \hdots, \]
\[ \varphi_{2} =  e^{1+\pi i/2} (w_{2} - w_{1}^{2} w_{2} w_{3}^{4}) + \hdots, \
\varphi_{3} = e^{1-\pi i/2} (w_{3} -  w_{1}^{2} w_{2}^{4} w_{3}) +\hdots  \]
The diffeomorphism $\varphi$ is real, hyperbolic, resonant and embedded in an analytic flow
by construction. There are $4$ non-zero weakly resonant monomials of $\varphi$ of degree $7$.
Notice that the non-linear monomial of lowest degree, i.e.
$e^{-2} w_{1}^{2} w_{2} w_{3} e_{1}$ is strongly resonant.
The next example shows that we can find weakly resonant monomials even at
the lowest degree.
\subsection{Example}
\label{subsec:ex2}
Let us consider a example with $X_{N}$ of degree $k=1$.
In particular $X_{s} + X_{N}$ is a real non-diagonalizable linear operator.
Then either all the eigenvalues of  $X_{s} + X_{N}$  are real (and the eigenvalues
are not weakly resonant) or $n \geq 4$.

Let us fix $n=4$. Consider the vector field
\[ X = 8 x_{1} \frac{\partial}{\partial x_{1}} + (x_{1} + 8 x_{2}) \frac{\partial}{\partial x_{2}} +
\left( x_{3} - \frac{\pi}{4} x_{4} \right)  \frac{\partial}{\partial x_{3}} +
\left( \frac{\pi}{4} x_{3} + x_{4} \right)  \frac{\partial}{\partial x_{4}} . \]
The Jordan Chevalley decomposition $X=X_{s} + X_{N}$ of $X$ is given by
$X_{N} = x_{1} \partial / \partial x_{2}$ and $X_{s}= X - X_{N}$.
The vector field $X$ is of the form
\[ X= 8 w_{1} \frac{\partial}{\partial w_{1}} + (w_{1} + 8 w_{2}) \frac{\partial}{\partial w_{2}} +
\left(1 +  \frac{\pi}{4} i \right)  w_{3} \frac{\partial}{\partial w_{3}} +
\left(1 -  \frac{\pi}{4} i \right)  w_{4}  \frac{\partial}{\partial w_{4}}  \]
in the coordinates $(w_{1}, \hdots, w_{4})$ introduced in subsection \ref{subsec:realmon}.
We have $\rho(1)=1$, $\rho(2)=2$, $\rho(3)=4$ and $\rho (4)=3$.
The list of weakly resonant monomials is
\[ w_{3}^{8} \frac{\partial}{\partial w_{1}}, \  w_{4}^{8} \frac{\partial}{\partial w_{1}}, \
w_{3}^{8} \frac{\partial}{\partial w_{2}} \ {\rm and} \ w_{4}^{8} \frac{\partial}{\partial w_{2}} . \]
We define the real vector field $Y' = (w_{3}^{8} + w_{4}^{8}) \partial / \partial w_{1}$.
We denote $\eta = {\rm exp}(Y')$.
We have
\[ \left[ Y' , X_{N} \right] =
\left[ (w_{3}^{8} + w_{4}^{8}) \frac{\partial}{\partial w_{1}} ,  w_{1} \frac{\partial}{\partial w_{2}} \right]
=  (w_{3}^{8} + w_{4}^{8}) \frac{\partial}{\partial w_{2}} . \]
We obtain
$\eta (w_{1},w_{2},w_{3},w_{4}) = (w_{1} + w_{3}^{8} + w_{4}^{8}, w_{2}, w_{3}, w_{4})$.
Formula (\ref{equ:taylorvf}) implies
$\eta^{*} X_{N} = (w_{1} + w_{3}^{8} + w_{4}^{8}) \partial / \partial w_{2}$.
We denote $\varphi = \eta^{\circ (-1)} \circ  {\rm exp}(X) \circ \eta$. We obtain
\[  \varphi  =
(e^{8} w_{1},  e^{8} w_{2}, e^{1+\frac{\pi}{4} i} w_{3}, e^{1- \frac{\pi}{4} i} w_{4}) \circ
(w_{1}, w_{1} + w_{2} + w_{3}^{8} + w_{4}^{8}, w_{3},w_{4})  \]
\[  \implies \varphi  =
(e^{8} w_{1}, e^{8} w_{1} + e^{8} w_{2} +e^{8}  w_{3}^{8} + e^{8} w_{4}^{8},
e^{1+\frac{\pi}{4} i} w_{3}, e^{1- \frac{\pi}{4} i} w_{4}) =  \]
\[   ( e^{8} z_{1}, e^{8} z_{1} + e^{8} z_{2} +
\frac{e^{8}}{2^{7}} \sum_{q=0}^{4} {8 \choose 2q} (-1)^{q} z_{3}^{2q} z_{4}^{8-2q},
e \sqrt{2} \frac{z_{3} - z_{4}}{2}, e \sqrt{2} \frac{z_{3} + z_{4}}{2}) .   \]
The diffeomorphism $\varphi$ is real, hyperbolic, resonant and embedded in an analytic flow
by construction. In spite of this all the non-linear monomials are weakly resonant.
Zhang's conjecture does not hold true in this case.
\subsection{Resonances as an obstacle to embed diffeomorphisms}
In spite of the previous examples we prove theorem \ref{teo:conj}.
It can be interpreted as a version of Zhang's
conjecture.

Let us discuss the optimality of the conditions in the theorem.
The examples
represent two different kind of obstructions to get a positive result.
The example in subsection \ref{subsec:ex1} satisfies $j^{1} X_{1,N} = 0$.
It is embeddable and it contains weakly resonant monomials immediately
above the lowest degree of non-linear
non-vanishing monomials. It does not satisfy (a) since
the lowest degree non-vanishing weakly resonant monomials have
degree $7$ but $f_{4} \neq 0$. It does not satisfy (b) either
since $w_{1} w_{3}^{3} e_{2}$ is a weakly resonant monomial of
degree $2 \leq 4 \leq 7-1$. Notice that $\varphi$
{\it does not have} weakly resonant monomials of degree $4$.
Weakly resonances of lower degree provide multiple choices for
the semisimple part of the embedding flow that can make the
diffeomorphism $\varphi$ to be embeddable.
Such an example justifies the need of restricting
our study to diffeomorphisms satisfying (a) or (b).

The existence of weakly resonant vector fields $Y$ with
$[X_{1,N}, Y] \neq 0$ allows to proceed as in the example \ref{subsec:ex2}
to obtain embeddabble diffeomorphisms having non-vanishing
weakly resonant monomials of the  lowest degree. The example in
section \ref{subsec:ex2} satisfies both (a) and (b).

The examples can be considered as a
classification of the type of counterexamples to the original conjecture.
\begin{proof}[proof of theorem \ref{teo:conj}]
Let $\hat{\varphi} \in \diffh{}{n}$ be the asymptotic development of $\varphi$.
Suppose that $\varphi = {\rm exp}(X)$ for some $X \in \Xrn{\infty}{n}$
(or $\Xrn{}{n}, \Xf{}{n}$) with $j^{1} X = X_{1}$.
Let $\hat{X} \in \Xf{}{n}$ be the asymptotic development of $X$.
We have $\hat{\varphi} = {\rm exp}(\hat{X})$. It suffices to prove
that $\hat{\varphi}$ is not embedded in a formal flow
$\hat{X}$ with $j^{1} \hat{X} = X_{1}$.

Consider the Jordan-Chevalley decomposition
$\hat{\varphi} = \varphi_{s} \circ \varphi_{u}$ of $\hat{\varphi}$.
We have $\varphi_{1,s}=j^{1} \varphi_{s}  = {\rm exp}(X_{1,s})$ and
$j^{1} \varphi_{u}  = {\rm exp}(X_{1,N})$. As a consequence
$\varphi_{1,s}=j^{1} \varphi_{s} $ is diagonal. We have
\[
\varphi_{s} = \varphi_{1,s} + A_{2} + A_{3}  + \hdots, \
\varphi_{u} = \varphi_{1,u} + B_{2} + B_{3} + \hdots
\]
where $A_{j}$ and $B_{j}$ are homogeneous of degree $j$ for $j \geq 2$.
The diffeomorphism $\varphi_{k-1}$ (see eq. (\ref{equ:actdif}))
commutes with $\varphi_{1,s}$ by hypothesis
(remark \ref{rem:resd}).
Since the Jordan-Chevalley decomposition
is compatible with the filtration in the space of jets we obtain
$j^{k-1} \varphi_{s} = \varphi_{1,s}$ or the equivalent property
$A_{2}=\hdots=A_{k-1}=0$.
Let us remark that all the non-vanishing monomials of $A_{k}$ are nonresonant;
otherwise $\varphi_{s}$ is not linearizable.
The vector field $X_{1,N}$ is strongly resonant since
$[X_{1,s},X_{1,N}]=0$ (remark \ref{rem:resvf}).
Therefore $\varphi_{1,u}$ is strongly resonant, it
preserves resonant and nonresonant polynomials. Then $A_{k} \circ \varphi_{1,u}$
is a sum of  nonresonant monomials. Since
$f_{k} = A_{k} \circ \varphi_{1,u} + \varphi_{1,s} \circ B_{k}$
the expression $\varphi_{1,s} \circ B_{k}$ has non-vanishing weakly resonant monomials.
The same property holds true for $B_{k}$.

Let $X_{1,N} + C_{2} + C_{3} + \hdots$ be the homogeneous decomposition of
$\log \varphi_{u}$. Since
$\varphi_{k-1} = \varphi_{1,s} \circ (\varphi_{1,s}^{\circ (-1)} \circ \varphi_{k-1})$
is the Jordan-Chevalley decomposition of $\varphi_{k-1}$
then $\varphi_{1,u} + B_{2} + \hdots + B_{k-1}$ is strongly resonant.
We obtain that $X_{1,N} + C_{2} + \hdots + C_{k-1}$ is strongly resonant by
using equation (\ref{equ:log}).
Let $B_{k}^{w}$ and  $C_{k}^{w}$ be the sum of the weakly resonant monomials of
$B_{k}$ and $C_{k}$ respectively.
It is easy to check out that
\[ {\rm exp}(X_{1,N} + C_{2} + \hdots + C_{k-1} + (C_{k} - C_{k}^{w}))
\]
is of the form $\varphi_{1,u} + B_{2} + \hdots + B_{k-1} + \sum_{j=k}^{\infty} \tilde{B}_{j}$
where $\tilde{B}_{k}$ is a sum of nonresonant and strongly resonant monomials.
As a consequence $B_{k}^{w} \neq 0$ implies $C_{k} \neq C_{k} - C_{k}^{w}$ and
$C_{k}^{w} \neq 0$.

Consider the  decomposition
$\hat{X} = X_{s} + X_{N}$ of $\hat{X}$.
Since ${\rm exp}(X_{s})$ is semisimple and ${\rm exp}(X_{N})$ is unipotent
we obtain $\varphi_{s} = {\rm exp}(X_{s})$ and $\varphi_{u} = {\rm exp}(X_{N})$.
We have $\log \varphi_{u} = X_{N}$ by lemma \ref{lem:exp}.

Suppose that (a) holds true. We obtain
$B_{2}=\hdots=B_{k-1}=0$ and then $C_{2}=\hdots=C_{k-1}=0$.
We denote
$X_{s} = X_{1,s} + X_{s}^{2} + X_{s}^{3} + \hdots$ where
$X_{s}^{j}$ is homogeneous of degree $j$ for any $j \geq 2$.
Let us calculate the degree $k$ component $D_{k}$ of $[X_{s}, \log \varphi_{u}]=0$.
We obtain
\[ 0 = D_{k} = [X_{1,s}, C_{k}] + [X_{s}^{k}, X_{1,N} ].  \]

Suppose that (b) holds true.
Since $j^{k-1} \varphi_{s} = \varphi_{1,s}$ and
$\varphi_{s}^{*} X_{s}=X_{s}$ we deduce that
$X_{s}^{2}$, $\hdots$, $X_{s}^{k-1}$ are resonant.
Condition (b) implies that they are also strongly resonant.
We claim that $X_{s}^{2} = \hdots = X_{s}^{k-1}=0$.
Otherwise $(X_{s})_{j} = (X_{s})_{j} + 0 = X_{1,s} + X_{s}^{j}$
are two different Jordan-Chevalley decompositions for
$j = \min \{ l \in \{2,\hdots,k-1\} : X_{s}^{l} \neq 0 \}$
(see equation (\ref{equ:actvf})).
Again we obtain
\[ 0 = D_{k} = [X_{1,s}, C_{k}] + [X_{s}^{k}, X_{1,N} ].  \]

The hypothesis on $X_{1,N}$ and lemma \ref{lem:monplay} imply that
$ [X_{s}^{k}, X_{1,N} ]$ does not contain non-vanishing weakly resonant monomials.
But clearly $[X_{1,s}, C_{k}]$ does since $C_{k}^{w} \neq 0$ (rem. \ref{rem:resvf}).
We obtain a contradiction.
\end{proof}
\begin{rem}
Let us remark that the condition on $X_{1,N}$ can be weakened. It is obvious
from the proof that it suffices to require $[X_{1,N},Y]=0$ for any
homogeneous weakly resonant vector field $Y$ of degree $k$.
\end{rem}
\begin{cor}
Let $A \in GL(3,{\mathbb R})$ and let $X_{1}$ be
a real logarithm of $A$ such that $X_{1,s}$ is diagonal.
Then any diffeomorphism $\varphi = Az + f_{k} + \hdots$ in
$\diffr{\infty}{n}$ (resp. $\diffrh{}{n}$)
such that $f_{k}$ contains
non-vanishing weakly resonant monomials is non-embeddable
 in the
flow of a vector field $X \in \Xrn{\infty}{n}$
(resp. $\Xr{}{n}$)
such that $j^{1} X = X_{1}$.
\end{cor}
This is a consequence that for $n=3$ either all the eigenvalues of $X_{1}$ are
real (and there are no weakly resonant monomials) or $X_{1,N} = 0$.
There exists a version of the result using property (b) instead of (a).
Notice that for $n=2$ if $A$ is hyperbolic and
has a real logarithm then $\varphi$  in
$\diffr{\infty}{n}$ is always embeddable in a $C^{\infty}$ flow \cite{Zhang}.

{\bf Example.} Consider the diffeomorphism $\varphi \in \diffr{}{3}$ defined by
\[ \varphi = \left(
e^{-2} z_{1}, -e z_{3} -\frac{3}{4} z_{1} z_{2} z_{3}^{2} + \frac{z_{1} z_{2}^{3}}{4},
e z_{2} + \frac{z_{1} z_{3}^{3}}{4}  -\frac{3}{4} z_{1} z_{2}^{2} z_{3}
\right) . \]
The eigenvalues of $j^{1} \varphi$ are $e^{-2}$, $e^{1+\pi i/2}$ and $e^{1-\pi i/2}$.
We consider the change of coordinates $z_{1}=w_{1}$,
$z_{2}=w_{2}+w_{3}$, $z_{3}=i(-w_{2}+w_{3})$ (see eq. (\ref{equ:rtoc})). We obtain
\[ (j^{1} \varphi)(w_{1},w_{2},w_{3}) = \left(
e^{-2} w_{1}, e^{1+\frac{\pi i}{2}} w_{2}, e^{1-\frac{\pi i}{2}} w_{3}
\right) \]
All eigenspaces of $j^{1} \varphi$ are one dimensional, hence any
logarithm $X_{1}$ of $j^{1} \varphi$ is diagonal and $X_{1,N}=0$.
The eigenvalues of $X_{1}$ are of the form
$\mu_{1}=-2 + 2 \pi i k_{1}$, $\mu_{2} = 1 + \pi i/2 + 2 \pi i k_{2}$ and
$\mu_{3} = 1 - \pi i/2 + 2 \pi i k_{3}$ for some
$k_{1},k_{2},k_{3} \in {\mathbb Z}$. We have
\[ \varphi(w_{1},w_{2},w_{3}) = \left( e^{-2} w_{1}, e^{1+\frac{\pi i}{2}} w_{2} + w_{1} w_{3}^{3},
e^{1-\frac{\pi i}{2}} w_{3} + w_{1} w_{2}^{3}
\right) . \]
The diffeomorphism $\varphi$ is resonant. Since
\[ (\mu_{1} + 3 \mu_{2} - \mu_{3}) - (\mu_{1} - \mu_{2} + 3 \mu_{3})
= 4 \pi i (1 + 2 (k_{2} -  k_{3}))  \]
either $w_{1} w_{3}^{3} e_{2}$ or $w_{1} w_{2}^{3} e_{3}$
is weakly resonant for any choice of $X_{1}$
(or equivalently for any choice of $k_{1},k_{2},k_{3} \in {\mathbb Z}$).
Condition (a) implies that $\varphi$ is a real hyperbolic
diffeomorphism that is not embedded in
a flow of $\Xrn{\infty}{n}$, $\Xrn{}{n}$, $\Xn{}{n}$ or $\Xf{}{n}$ (th. \ref{teo:conj}).
\bibliography{rendu}
\end{document}